\documentclass[11pt,a4paper,reqno]{amsart}
\usepackage{amsmath}
\usepackage{amsfonts}
\usepackage{amssymb}
\usepackage{amscd}
\usepackage{url}
\usepackage{enumerate}
\usepackage[pdftex,bookmarks=true]{hyperref}
\usepackage{color}

\renewcommand\Re{{\operatorname{Re}}}

\newcommand\rank{{\operatorname{rank}}}

\newcommand\R{{\mathbf{R}}}
\newcommand\C{{\mathbf{C}}}
\newcommand\M{{\mathbf{M}}}

\renewcommand\P{{\mathbf{P}}}
\newcommand\E{{\mathbf{E}}}

\newcommand\Z{{\mathbf{Z}}}

\newcommand\tr{{\operatorname{tr}}}

\newcommand\per{{\operatorname{per}}}

\newcommand\ep{\varepsilon}

\newcommand\Ba{{\mathbf a}}

\newcommand\Bc{{\mathbf c}}

\newcommand\Be{{\mathbf e}}

\newcommand\Bg{{\mathbf g}}

\newcommand\Bu{{\mathbf u}}
\newcommand\Bv{{\mathbf v}}

\newcommand\Bx{{\mathbf x}}
\newcommand\By{{\mathbf y}}
\newcommand\Bz{{\mathbf z}}

\renewcommand\Pr{{\mathbf P }}

\newcommand\CC{{\mathcal C}}

\newcommand\CF{{\mathcal F}}

\newcommand\CI{{\mathcal I}}

\renewcommand\O{\mathbf{O}}

\newcommand\col{{\mathbf{c}}}
\newcommand\row{{\mathbf{r}}}
\newcommand\0{{\mathbf{0}}}

\theoremstyle{plain}
  \newtheorem{theorem}[subsection]{Theorem}
  \newtheorem{conjecture}[subsection]{Conjecture}

  \newtheorem{proposition}[subsection]{Proposition}
  
  \newtheorem{lemma}[subsection]{Lemma}
  \newtheorem{corollary}[subsection]{Corollary}
  \newtheorem{question}[subsection]{Question}
  \newtheorem{example}[subsection]{Example}

  \newtheorem{claim}[subsection]{Claim}

\theoremstyle{definition}
  \newtheorem{definition}[subsection]{Definition}

\author{Scott Aaronson}
\email{aaronson@csail.mit.edu}
\address{Department of Electrical Engineering and Computer Science, Massachusetts Institute of Technology, Cambridge, MA 02139}

\author{Hoi Nguyen}
\email{nguyen.1261@math.osu.edu}
\address{Department of Mathematics, The Ohio State University, Columbus, OH 43210}

\thanks{S.\ Aaronson is supported by an NSF Waterman Award.  H.\ Nguyen is supported by research grant DMS-1358648}

\include{psfig}

\begin{document}

\title{Near invariance of the hypercube}

\maketitle

\begin{abstract} We give an almost-complete description of orthogonal matrices $M$ of order $n$ that ``rotate a non-negligible fraction of the Boolean hypercube $\CC_n=\{-1,1\}^n$ onto itself,'' in the sense that
$$\P_{\Bx\in \CC_n}(M\Bx\in \CC_n) \ge n^{-C},\mbox{ for some positive constant } C,$$
where $\Bx$ is sampled uniformly over $\CC_n$. \ In particular, we show that such matrices $M$ must be very close to products of permutation and reflection matrices. \ This result is a step toward characterizing those orthogonal and unitary matrices with large permanents, a question with applications to linear-optical quantum computing.

\end{abstract}

\section{Introduction}\label{section:intro}

Let $M=(m_{ij})_{1\le i,j\le n}$ be a square matrix of order $n$ of real entries. \ Motivated by a question from linear-optical quantum computing (see \cite{AA}), the first named author and Hance \cite{AH} asked the following question.

\begin{question}\label{question:perm} Characterize all matrices $M$  such that $\|M\|_2\le 1$ and there exists a constant $C\ge 0$ such that
$$\per(M) \ge n^{-C}.$$
\end{question}

It is not hard to show that (see also \cite{G}), with $\Bx=(x_1,\dots,x_n)$ where $x_1,\dots,x_n$ are iid copies of an arbitrary real random variable $\xi$ of mean zero and variance one,

$$\per(M) = \E_{\Bx} \prod_{i=1}^n x_i (M\Bx)_i.$$

Thus, if we choose $\xi$ to be the Bernoulli random variable (taking values $\pm 1$ independently with probability 1/2), then $\per(M)\ge n^{-C}$ would imply that

$$\E_{\Bx\in \CC_n} \prod_{i=1}^n  |(M\Bx)_i| \ge n^{-C}.$$

where $\CC_n$ denotes the $n$-dimensional hypercube $\{-1,1\}^n$,  and $\Bx$ is chosen uniformly in  $\CC_n$

Furthermore, as $ \prod_{i=1}^n  |(M\Bx)_i| \le 1$, a simple calculation gives

\begin{equation}\label{eqn:scorefn}
s(M):=\P_{\Bx\in \CC_n}\left(\prod_{i=1}^n |(M\Bx)_i |\ge n^{-C}/2\right) \ge n^{-C}/2.
\end{equation}

For short, the quantity $s(M)$ above is called the {\it score function} of $M$. \ Equation \eqref{eqn:scorefn} motivates us to ask the following question (see also \cite{A}.)

\begin{question}\label{question:score} Characterize all matrices $M$ such that $\|M\|_2\le 1$ and there exists a constant $C$ such that $s(M)\ge n^{-C}/2$.
\end{question}

As the direct study of $s(M)$ seems very difficult at the moment, the goal of this paper is to focus on a simpler (but closely related) object as follows. \ We define the {\it exact score function} of $M$ to be

$$s_0(M):=\P_{\Bx\in \CC_n}(M\Bx \in \CC_n).$$

Clearly, $s_0(M)\le s(M)$.  \ We observe that, as $s_0$ measures how $\CC_n$ is preserved under $M$, or how far the random vector $M\Bx$ is from being a Bernoulli vector, the study of this exact score function is natural on its own.

For the sake of discussion, we will be focusing on orthogonal matrices for the rest of this section. Observe that if $M\Bx = (\ep_1 x_{\pi(1)},\dots, \ep_n x_{\pi(n)})$  for any choice of signs $\ep_i \in \{-1,1\}$, and for any permutation $\pi$ in $S_n$, then $s_0=1$. \ In other words, if $M$ is a product of permutation and reflection matrices (or shortly permutation-reflection matrices), then $s_0(M) =1$. \ We would like to study the following inverse problem.


\begin{question}\label{question:0}
Are permutation-reflection matrices ``essentially'' the only orthogonal matrices with large $s_0(M)$, say $s_0(M)\ge n^{-C}$ for some positive constant $C$?
\end{question}

In this paper, we confirm this heuristic by showing the following.

\begin{theorem}[Main application]\label{theorem:orthogonal:2} Let $0<\ep<1$ and $C>0$ be given constants. \ Assume that $M\in \O(n)$ with $s_0(M)\ge n^{-C}$. \ Then all but $O(n^{1-\ep})$ rows of $M$ contain a (unique) entry of absolute value at least $1-O(n^{-1+\ep})$.
\end{theorem}

We will see in Example \ref{example:3} that the lower bound $1-O(n^{-1+o(1)})$ on the large entries is tight. \ We believe that our approach to proving Theorem \ref{theorem:orthogonal:2}, which goes through  inverse Littlewood-Offord theory initiated by Tao and Vu, will also be useful for the study of Question \ref{question:score}.

In Appendix \ref{section:stochastic}, we answer Question \ref{question:perm}, but for \textit{stochastic} matrices, rather than orthogonal matrices or matrices of bounded norm. \ In more detail, we show there that, if $A$\ is an $n\times n$\ stochastic matrix with $\operatorname*{Per}\left(  A\right)  \geq n^{-O\left(  1\right)  }$, then all but $O\left(  \log n\right)  $\ of the rows of $A$\ are dominated by a single large entry (and in that sense, $A$ is ``close to a permutation matrix'').

Let us mention a few interesting applications and alternative statements of
Theorem \ref{theorem:orthogonal:2}.

First, Theorem \ref{theorem:orthogonal:2} implies that there is no $n\times n$\ orthogonal matrix
$M$\ that maps a non-negligible fraction of \textit{uniform superpositions} to
other uniform superpositions, besides \textquotedblleft highly
degenerate\textquotedblright\ matrices (i.e., those close to permuted diagonal
matrices). \ Here a uniform superposition is defined to be any quantum state
of the form%
\[
\left\vert \psi\right\rangle =\frac{\pm\left\vert 1\right\rangle \pm\cdots
\pm\left\vert n\right\rangle }{\sqrt{n}}.
\]
These states often arise and are of interest in quantum computing, and
a-priori, one might have hoped that there would be interesting transformations
that had a non-negligible probability of staying within the set of such
states. \ We conjecture that an analogous result should hold for arbitrary
\textit{unitary} matrices, except

\begin{enumerate}
\item with the condition that $M\left\vert \psi\right\rangle $ is a uniform
superposition relaxed to the condition that $\left\vert \left\langle
i\right\vert M\left\vert \psi\right\rangle \right\vert =1/\sqrt{n}$\ for all
$i\in\left[  n\right]  $, and
\vskip .1in
\item with the exception for matrices close to permuted diagonal matrices
broadened to include matrices close to permuted \textit{block}-diagonal
matrices, with $2\times2$\ blocks such as%
\[
B=\frac{1}{\sqrt{2}}\left(
\begin{array}
[c]{cc}%
1 & i\\
1 & -i
\end{array}
\right)  .
\]
For one can check that $B$ maps each of the four vectors $\left(  1,1\right)
$, $\left(  1,-1\right)  $, $\left(  -1,1\right)  $, $\left(  -1,-1\right)  $
to a vector both of whose entries have equal magnitude. \ (We thank Sumegha
Garg for this observation.)
\end{enumerate}

Second, it is clear that $\Bx^{T}M\By\leq n$, if $M$ is an $n\times n$\ orthogonal
matrix and $\Bx$ and $\By$\ are any vectors in $\CC_n$.
\ Moreover it is clear that%
\[
\Pr_{\Bx,\By\in\CC_n}\left(\Bx^{T}M\By=n\right) \leq\frac
{1}{2^{n}},
\]
since we can only have equality if $x=My$. \ Theorem \ref{theorem:orthogonal:2} says that, unless $M$
is highly degenerate (in particular, unless all but $O\left(  n^{1-\varepsilon
}\right)  $\ of $M$'s rows are dominated by a single large entry), we have%
\[
\Pr_{\Bx,\By\in\CC_n}\left(\Bx^{T}M\By=n\right)  =o\left(
\frac{1}{2^{n}n^{C}}\right)
\]
for all constants $C$.

We thank Alex Arkhipov for the third application of Theorem \ref{theorem:orthogonal:2}. \ Given two
points $\Bx,\By\in \CC_n$, let their \textit{Hamming
distance} $\Delta\left(  \Bx,\By\right)  $\ be the number of coordinates on which
they differ. \ Then given a subset $S\subseteq \CC_n$,
and a function $f:S\rightarrow \CC_n$, we call $f$ a
\textit{bi-Lipschitz bijection of} $S$ if%
\[
\Delta\left(  f\left(  \Bx\right)  ,f\left(  \By\right)  \right)  =\Delta\left(
\Bx,\By\right)
\]
for all $\Bx,\By\in S$. \ Clearly, for any $S$, we can produce $n!2^{n}%
$\ different bi-Lipschitz bijections of $S$\ by simply permuting and
reflecting the $n$ coordinates of the hypercube. \ However, one might wonder
for which $S$'s there exist bi-Lipschitz bijections that are more interesting
than that. \ We claim that Theorem \ref{theorem:orthogonal:2} implies that, \textit{if }$\left\vert
S\right\vert \geq2^{n}/n^{O\left(  1\right)  }$\textit{, then any bi-Lipschitz
bijection of }$S$\textit{ is close (in some sense) to a permutation and
reflection of the coordinates.} \ This is a consequence of the following proposition.

\begin{proposition}\label{prop:BL}
Given any bi-Lipschitz bijection $f:S\rightarrow \CC_n$,
there exists an orthogonal matrix $M$ such that $M\Bx=f\left(  \Bx\right)  $\ for
all $\Bx\in S$.
\end{proposition}

\begin{proof}(of Proposition \ref{prop:BL})
Note that, if we interpret $\Bx,\By\in \CC_n$\ as points in
$\mathbb{\R}^{n}$, then

$$\Delta\left(  \Bx,\By\right)  =\frac{\|
\Bx-\By\|_2^{2}}{4}.$$

Thus,%
\[
\|f(\Bx)-f(\By)\|_2 = \|\Bx-\By\|_2
\]
for all $\Bx,\By\in S$. \ Clearly we also have $ \|f(\Bx)\|_2 =\|\Bx\|_2 (=\sqrt{n})$\ for all $\Bx\in S$. \ This means that the set
$f\left(  S\right)  $\ is a rigid rotation and/or reflection of the set $S$,
so it must be possible to get from one to the other by applying an orthogonal matrix.
\end{proof}


\section{Characterization of matrices with large score function}\label{section:lemmas}

In line with Question \ref{question:perm} and \ref{question:score}, it is natural to study the exact score function for more general matrices.

\begin{question}\label{question:1}
Assume that $M\in \M(n)$ with $s_0(M) \ge n^{-C}$ for some constant $C>0$. \ What can we say about the structure of $M$?
\end{question}

Let us consider some natural candidates for $M$.

\begin{example}[special matrices of $\{-1,0,1\}$ entries] \label{example:1}
It is clear that if $M$ is a $\{-1,0,1\}$ matrix where each row contains exactly one non-zero entry, then $s_0(M) =1$. \ More generally, one can construct $M$ satisfying $s_0(M) \ge n^{-C}$ from such matrices of size close to $n$ together with other block matrices of extremely small size.
\end{example}


For further examples, we introduce the notion of generalized arithmetic progression (GAP).

\begin{definition}
A set $Q\subset Z$, where $Z$ is an abelian torsion-free group, is a \emph{GAP of rank $r$} if it can be expressed in the form
$$Q= \{\Bg_0+ k_1\Bg_1 + \dots +k_r \Bg_r : K_i \le k_i \le K_i', k_i\in \Z \hbox{ for all } 1 \leq i \leq r\}$$ for some $\Bg_0,\ldots, \Bg_r\in Z$, and some integers $K_1,\ldots,K_r,K'_1,\ldots,K'_r$.

The elements $\Bg_i\in Z$ are the \emph{generators} of $Q$, the numbers $K_i'$ and $K_i$ are the \emph{dimensions} of $Q$. \ We say that $Q$ is \emph{proper} if $|Q| = \prod (K_i'-K_i+1)$. \ If $g_0=0$ and $-K_i=K_i'$ for all $i\ge 1$, we say that $Q$ is {\it symmetric}.
\end{definition}

\begin{example}[Additively perturbed matrices] \label{example:2} One can perturb a matrix from Example \ref{example:1} by elements from a GAP to obtain a matrix $M\in \M(n)$ with large score function. \ Indeed, let $F_0$ be any matrix of size $n$ from Example \ref{example:1}, and let $\Bg_1,\dots,\Bg_r \in \R^n$ be $r$ vectors in $\R^n$, where $r=O(1)$. \ Consider a GAP

$$Q = \{k_1\Bg_1 + \dots +k_r \Bg_r : |k_i| \le K_i\},$$

where $\prod_{i=1}^r (2K_i+1) =n^{O(1)}$.

Choose any $n-1$ elements $\Bu_1,\dots, \Bu_{n-1}$ from $Q$. \ By a standard deviation principle and by the pigeonhole principle, there exists $\Bu_n \in (10\sqrt{n})Q = \{k_1\Bg_1 + \dots +k_r \Bg_r : |k_i| \le 10\sqrt{n} K_i\}$ such that

$$\P_{\Bx\in \CC_n}(\sum_{i=1}^{n-1} x_i\Bu_i+\Bu_n=0)\ge n^{-O(1)}.$$

Let $U$ be the matrix of the column vectors $\Bu_i$ and set $M:=F_0+U$. \ By definition,

$$s_0(M)=\P_{\Bx\in \CC_n}(M\Bx\in \CC_n)\ge \P_{\Bx\in \CC_n}(U\Bx=0) \ge \P_{\Bx\in \CC_n}\left(\sum_{i=1}^{n-1} x_i\Bu_i+\Bu_n=0\right)\ge n^{-O(1)}.$$
\end{example}

It is thus natural to conjecture that the matrices from Example \ref{example:1} and Example \ref{example:2} are essentially the only ones that have large score function. \ We support this conjecture by showing the following.

\begin{definition}
For integers $a,b$, let $\CF_{ab}$ denote the collection of all $\{-1,0,1\}$-matrices of size $a \times b$ where each row contains {\it at most} one non-zero entry.
\end{definition}

\begin{theorem}[Characterization of general matrices, main result]\label{theorem:discrete:main:1}Let $0<\ep<1$ and $C$ be positive constants. Suppose that $M=(m_{ij})_{1\le i,j\le n}\in \M(n)$ satisfies $s_0(M) \ge n^{-C}$ for some positive constant $C>0$. \ Then there exists a submatrix $M'$ of $M$ of size $n_1\times n_2$, with $n_1,n_2=n-O_{C,\ep}(n^{1-\ep})$, and a set of $r=O_{C,\ep}(1)$ vectors $\Bg_1,\dots,\Bg_r\in \R^{n_1}$ such that $M'$ can be written as

$$M'= M''+F,$$

where $F\in \CF_{n_1n_2}$, and the columns of $M''$ belong to a GAP of size $n^{O_{C,\ep}(1)}$ generated by $\Bg_1,\dots,\Bg_r$.
\end{theorem}

When $M$ has nearly full rank, one can easily deduce the following consequence with much more information on $F$.

\begin{corollary}\label{cor:fullrank} Suppose that $M$ satisfies the condition of Theorem \ref{theorem:discrete:main:1} as well as $\rank(M)=n-o(n)$. \ Then one can also assume that each row of $F$ contains exactly one non-zero entry which is either $1$ or $-1$. \ In other words, $F$ is nearly a permutation-reflection matrix modulo a low-rank perturbation.
\end{corollary}

We next focus on orthogonal matrices. \ Similarly to Question \ref{question:1}, one would like to characterize orthogonal matrices with large score function.

\begin{question}\label{question:2}
Assume that $M$ is an orthogonal matrix satisfying $ s_0(M)\ge n^{-C}$ for some constant $C>0$. \ What can we say about the structure of $M$?
\end{question}

As suggested by Theorem \ref{theorem:discrete:main:1} and Corollary \ref{cor:fullrank}, it is natural to search for $M$ from the matrices of Example \ref{example:2}: thus are there non-trivial low rank perturbations of an orthogonal matrix which are again orthogonal and  $s_0(M)\ge n^{-C}$? \ The answer is positive, even for rank-one perturbation.

\begin{example}\label{example:3}
Let $\Bu_1=(x,t_2x,\dots,t_nx) \in \R^n$, where $x\neq 0$ and $t_2,\dots, t_n\in \R$ are to be chosen. \ We will select $\Bu_1$ so that $\col_1=\Bu_1+\Be_1$ has norm 1, which hence requires

\begin{equation}\label{eqn:row1:useful}
(x+1)^2 +x^2(t_2^2+\dots+t_n^2)=1, \mbox{ or equivalently, }  \Bu_1 \cdot \Bu_1+ 2x =0.
\end{equation}

For $2\le i\le n$, define $\Bu_i:=t_i \Bu_1$. \ Thus if  $t_i=k_i/N$ with $k_i\in \Z, |k_i|\le n^{O(1)}$ then $\Bu_i$ belongs to the rank-one arithmetic progression $\{k  \Bu_1/N, |k|\le n^{O(1)}\}$.

Set $\col_i:= \Bu_i+\Be_i$, and let $M$ be the matrix of the column vectors $\Bc_i$. \ In other words,

\begin{equation}\label{eqn:M:rank1}
M=(\delta_{ij}+xt_it_j)_{1\le i,j\le n}.
\end{equation}

By definition, it can be verified that

$$\|\col_i\|_2=1, \mbox{ and }\col_i \cdot \col_j =0 , \forall i\neq j.$$

Indeed, for the first part, using \eqref{eqn:row1:useful}

$$\|\col_i\|_2^2 = \Bu_i \cdot \Bu_i + 2 \Bu_i \cdot \Be_i + 1= t_i^2 \Bu_1 \cdot \Bu_1 + 2t_i^2 x +1=1.$$

For the second part, similarly,

$$\col_i \cdot \col_j = \Bu_i \cdot \Bu_j + \Bu_i \cdot \Be_j + \Bu_j \cdot \Be_i = t_it_j \Bu_1 \cdot \Bu_1+ 2xt_it_j=0.$$

Notice that the specific choice of $x=-2/n$ and $t_i=1$ would fulfill \eqref{eqn:row1:useful}, and hence confirms the asymptotic sharpness of Theorem \ref{theorem:orthogonal:2}.

In what follows we describe another natural way to sample the $t_i$.

\begin{itemize}
\item First, select integers $k_2,\dots,k_n\in \Z$ within the range $|k_i| \le n^{O(1)}$ so that  $\P_{\Bx\in \CC_n}(x_1+ \sum_{i=2}^n x_ik_i=0)\ge n^{-O(1)}$. \ (Although this is not true for all choices of $k_i$, it holds for many natural families.)
\vskip .1in
\item Next, choose $x=-1$ and set $t_i:=k_i/N$, where $N^2:=\sum_{i=2}^n k_i^2$.
\end{itemize}

It is clear that \eqref{eqn:row1:useful} is fulfilled, and hence by definition,

$$s_0(M)=\P_{\Bx\in \CC_n}(M\Bx\in \CC_n)\ge \P_{\Bx\in \CC_n}(U\Bx=\0) \ge n^{-O(1)}.$$

\end{example}

As an extension of Theorem \ref{theorem:discrete:main:1} and Corollary \ref{cor:fullrank}, we show that if $M$ is as in Question \ref{question:2}, then it essentially resembles the matrices of Example \ref{example:3}.

\begin{theorem}[Characterization of orthogonal matrices, main result]\label{theorem:discrete:main:2} Let $0<\ep<1$ and $C$ be positive constants. \ Suppose that $M=(m_{ij})_{1\le i,j\le n}\in \O(n)$ satisfies $s_0(M) \ge n^{-C}$ for some positive constant $C>0$. \ Then there exists a submatrix $M'$ of $M$ of size $n\times n_2$, with $n_2=n-O_{C,\ep}(n^{1-\ep})$, and a set of $r=O_{C,\ep}(1)$ vectors $\Bg_1,\dots,\Bg_r\in \R^{n}$ such that $M'$ can be written as

$$M'=M''+F,$$

where $F\in \CF_{nn_2}$ and the columns of $M''$ belong to a GAP of small size generated by $\Bg_1,\dots,\Bg_r$.

Furthermore, the matrix $M''+F$ (modulo appropriate row permutations) when restricting to the first $n_2$ rows can be written as

\begin{equation}\label{eqn:M:rankr}
 \begin{pmatrix}
U & UD^T\\
DU  & DUD^T
\end{pmatrix}
+ I
\end{equation}

where $U$ is a square matrix of size $r$ and $D$ is an $(n_2-r)\times r$ matrix, and $I$ is a diagonal square matrix of size $n_2$ of entries from $\{-1,1\}$.
\end{theorem}


We remark that \eqref{eqn:M:rankr} is a generalization of \eqref{eqn:M:rank1}, and it is not hard to construct small rank perturbations of type \eqref{eqn:M:rankr} of permutation-reflection matrices which are again orthogonal and have large score function.

One of the main tools of our treatment is an inverse-type Littlewood-Offord result, which was proved in \cite{TV1,TV2} by Tao and Vu. \ Here we state a version from \cite[Theorem 2.5]{NgV}.

\begin{theorem}\label{theorem:optimal} Let $0<\ep<1$ and $C$ be positive constants. \ Assume that $\Ba_1,\dots,\Ba_n$ are elements of an abelian torsion-free group $Z$ such that

$$\rho(\Ba_1,\dots,\Ba_n):=\sup_{\Ba\in Z} \P_{\Bx\in \CC_n}\left(\sum_{i=1}^n x_i \Ba_i =\Ba\right)  \ge  n^{-C}.$$

Then there exists a proper symmetric GAP $Q\subset Z$ which contains all but at most $n^\ep$
elements of $\Ba_1,\dots,\Ba_n$. \ Furthermore,

\begin{itemize}
\item the rank $r$ of $Q$ is bounded by a constant, $r= O_{C, \ep} (1)$,
\vskip .1in
\item the size of $Q$ is small, $|Q| = O_{C, \ep} (\rho^{-1} n^{-\ep r/2}).$
\end{itemize}
\end{theorem}

It is crucial to remark that the implied constants are independent of the group $Z$. \ Furthermore, by following the proof of \cite[Theorem 2.5]{NgV}, these constants are $O\left(2^{2^{2^{\alpha(C/\ep)}}}\right)$ at most, where $\alpha$ is an absolute constant. \ Consequently, the dependent constants under $O_{C,\ep}(.)$ in Theorem \ref{theorem:discrete:main:1} and Theorem \ref{theorem:discrete:main:2} can be taken to be $O\left(2^{2^{2^{\alpha (C/\ep) }}}\right)$. \ However, as these dependencies are not our current focus, we will skip the details to ease the presentation.


{\bf Notation.}
For a matrix $M$, $\row_i(M), \col_j(M)$ denote the $i$-th row and $j$-th column respectively. \ For a vector $\Bv\in \R^n$, $(\Bv)_i$ denotes its $i$-th component. \ For an index set $T\subset [n]$, $\Bv^{[T]}$ represents the subvector of $\Bv$ of components indexing in $T$.

For short, we say that a symmetric proper GAP $P=\{\sum_{i=1}^r k_i \Bg_i\}$ has small size and bounded rank if $|P|=n^{O(1)}$ and $\rank(P)=O(1)$. \ We say that $r$ elements $\Bx_1=k_{11}\Bg_1+\dots,k_{1r}\Bg_r, \dots, \Bx_r=k_{r1}\Bg_1+\dots,k_{rr}\Bg_r$ span $P$ if the corresponding vectors $$(k_{11},\dots,k_{1r}),\dots, (k_{r1},\dots,k_{rr})$$ have full rank in $\R^r$.

The rest of the paper is organized as follows. \ We will introduce some key lemmas in Section \ref{section:lemmas}, and prove Theorem \ref{theorem:discrete:main:1} and Theorem \ref{theorem:discrete:main:2} in Section \ref{section:general} and Section \ref{section:orthogonal} respectively. \ Before concluding the paper with a problem section (Section \ref{section:problem}) and a remark on stochastic matrices (Appendix \ref{section:stochastic}), we  give two applications: one shows that general matrices with sufficiently small entries are not near invariant with respect to the hypercube (Section \ref{section:application:1}), and one deduces Theorem \ref{theorem:orthogonal:2} (Section \ref{section:application:2}).


\section{Structural relation between rows and columns}\label{section:lemmas}

Assume that $M\in \M(n)$ satisfies the condition $s_0(M)\ge n^{-C}$ of Theorem \ref{theorem:discrete:main:1}. \ By Theorem \ref{theorem:optimal}, it is not hard to show that for any row $\row_i=(m_{i1},\dots,m_{in})$ of $M$, all but $n^\ep$ of the components $m_{ij}$ belong to a GAP of bounded rank and small size. \ In the lemma below we slightly improve this result for collections of several rows.

For any $d$ indices $1\le i_1<\dots < i_d\le n$, let $\col_1^{[i_1,\dots,i_d]},\dots,\col_n^{[i_1,\dots,i_d]}$ be the column vectors of the $d\times n$ submatrix spanned by $\row_{i_1}(M),\dots, \row_{i_d}(M)$ of $M$. \ We will prove the following.

\begin{lemma}[Structure for row and column vectors I. ]\label{lemma:drows}  Let $0<\ep<1$ and $C$ be positive constants. \ Let $M$ be a matrix  with $s_0(M)\ge n^{-C}$. Assume that $1\le d \le c\log n$, where $c>0$ is a  constant, and let $1\le i_1<\dots <i_d\le n$ be any $d$ indices. \ Then there exist an exceptional index set $I_{i_1,\dots,i_d}\subset [n]$ of size at most $n^\ep$, and a GAP $Q_{i_1,\dots i_d}\subset \R^d$ which contains all $\col_i^{[i_1,\dots,i_d]}, i\in \bar{I}_{i_1,\dots,i_d}$, and such that

$$|Q|=O(2^d n^C/n^{r\ep/2}).$$

As a consequence,

\begin{enumerate}[(i)]
\item the rank $r$ of $Q$ is bounded, $\rank(Q)=O(1)$;
\vskip .1in
\item the rows of the matrix generated by $\col_i^{[i_1,\dots,i_d]}, i\in \bar{I}_{i_1,\dots,i_d}$, span a subspace of dimension at most $\rank(Q)$.
\end{enumerate}
\end{lemma}

\begin{proof}(of Lemma \ref{lemma:drows}) As $\P_{\Bx\in \CC_n}(M\Bx\in \CC_n)\ge n^{-C}$, by applying the projection $\pi_{i_1,\dots,i_d}$ (mapping $\R^n$ onto the components of indices $i_1,\dots,i_d$), the column vectors $\col_1^{[i_1,\dots,i_d]},\dots,\col_n^{[i_1,\dots,i_d]}$ have large concentration probability $\rho$,

$$\rho(\col_1^{[i_1,\dots,i_d]},\dots,\col_n^{[i_1,\dots,i_d]}) \ge n^{-C}2^{-d}.$$

The conclusion of Lemma \ref{lemma:drows} then follows by applying Theorem \ref{theorem:optimal} to these column vectors.
\end{proof}

We next show that the result of Lemma \ref{lemma:drows} can be extended to collections of as many as $n^{1-\ep}$ rows. \ Let $r_0=O_{C,\ep}(1)$ be an upper bound for all $\rank(Q)$ (applied to all $d$ indices $i_1,\dots,i_d$) from Lemma \ref{lemma:drows}.


\begin{lemma}[Structure for row  and column vectors II. ]\label{lemma:mrows}  Let $1\le k\le n^{1-\ep}$ and $r_0\le d\le c\log n$. \ Consider any $m$ indices $1\le i_1<\dots <i_m\le n$, where $m=k(d-r_0)+d$. \ Then for the rows $\row_{i_1},\dots,\row_{i_m}$ of $M$, there exist an exceptional index set $I_{i_1\dots i_m}\subset [n]$ of size at most $(k+1)n^{\ep}$ and a GAP $Q_{i_1,\dots, i_m}\subset \R^m$ of rank $r\le r_0$ such that the following holds.

\begin{enumerate}[(i)]
\item $Q_{i_1,\dots, i_m}\subset \R^m$ contains all $\col_i^{[i_1,\dots,i_m]}, i\in \bar{I}_{i_1,\dots,i_m}$, and

$$|Q_{i_1,\dots,i_m}|=O(2^d n^C/n^{r\ep/2}).$$

\item The row vectors of the submatrix spanned by $\col_i^{[i_1,\dots,i_m]}, i\in \bar{I}_{i_1,\dots,i_m}$, span a subspace of dimension at most $r_0$.

\end{enumerate}
\end{lemma}

\begin{proof}(of Lemma \ref{lemma:mrows})
For (ii), we first apply Lemma \ref{lemma:drows} to the first $d$ rows $\row_{i_1},\dots, \row_{i_d}$ to obtain an exceptional set $I_{i_1,\dots,i_d}$ and a collection of $r_0$ rows which span the subspace of the rows of the matrix generated by $\col_{i}^{[i_1,\dots,i_d]}, i\in \bar{I}_{i_1,\dots, i_d}$. \ We next add to these $r_0$ rows another $d-r_0$ ones among the remaining $\row_{i_{d+1}},\dots,\row_{i_m}$ and apply Lemma \ref{lemma:drows} again. \ Iterate this process $k+1$ times and let $I=I_{i_1,\dots, i_m}$ be the union of the exceptional sets from each step. Hence $|I|\le (k+1) n^\ep$. \ One can check that by definition the row vectors of the matrix spanned by $\col_i^{[i_1,\dots,i_m]}, i\in \bar{I}_{i_1,\dots,i_m}$, span a subspace of dimension $r\le r_0$.

For (i), we will show that the  column vectors $\col_i^{[i_1,\dots,i_m]}$ belong to a GAP by restricting to the projection of $\R^n$ onto the components of indices from $\bar{I}_{i_1,\dots,i_m}$. \ At the last step of the above process, choose $r\le r_0$ rows $\row_{j_1},\dots, \row_{j_{r}}$ that span the subspace generated by the lastly considered $d$ rows, and hence by definition of the process, they also span the subspace generated by all $m$ rows (restricting to the components of indices from $\bar{I}_{i_1,\dots,i_m}$).

After the application of Lemma \ref{lemma:drows} at this last step, the columns $\col_{i}^{[j_1,\dots,j_{r}]}, i\in \bar{I}_{i_1,\dots,i_m}$, of the matrix corresponding to $\row_{j_1},\dots, \row_{j_{r}}$ belong to a GAP $Q_{r}\subset \R^{r}$ of rank $r\le r_0$ and of small size. \ Let $\Bg_1^{[j_1,\dots,j_r]},\dots, \Bg_r^{[j_1,\dots,j_r]} \in \R^r$ be the generators of $Q_r$. \ We claim that this structure can be extended to $\R^m$ to contain all the extension vectors $\col_{i}^{[i_1,\dots,i_m]}$ of $\col_{i}^{[j_1,\dots,j_r]}$. \ Recall that

\begin{enumerate}
\item As by (ii), for any $i\in \{i_1,\dots,i_m\}$, there exist real coefficients $\alpha_{i1},\dots, \alpha_{ir} \in \R$ such that the restricted row vector $\row_i^{[\bar{I}_{i_1,\dots,i_m}]} \in \R^{n-|I_{i_1,\dots,i_m}|}$ satisfies

$$\row_i^{[\bar{I}_{i_1,\dots,i_m}]}= \alpha_{i1} \row_{j_1}^{[\bar{I}_{i_1,\dots,i_m}]} + \dots+ \alpha_{ir} \row_{j_r}^{[\bar{I}_{i_1,\dots,i_m}]}.$$

\vskip .1in

\item  For any $i \in \bar{I}_{i_1,\dots,i_m}$,  there exist integral coefficients $k_{i1},\dots, k_{ir} \in \R$ with $|k_{i1}\dots k_{ir}| \le |Q_r|$ such that the column vector $\col_i^{[j_1,\dots,j_r]}$ satisfies

$$\col_i^{[j_1,\dots,j_r]} =k_{i1} \Bg_{1}^{[j_1,\dots,j_r]} + \dots + k_{ir} \Bg_{r}^{[j_1,\dots,j_r]}.$$

\end{enumerate}

\begin{claim}[Structure extension]\label{claim:extension}
There exists a GAP $Q_m$ in $\R^m$ that has the same rank and size as $Q_r$ which contains the extension $\col_{i}^{[i_1,\dots,i_m]}$ of $\col_{i}^{[j_1,\dots,j_r]}$, where $i\in \bar{I}_{i_1,\dots,i_m}$,.
\end{claim}

\begin{proof}(of Claim \ref{claim:extension}) First, by using the coefficients $\alpha_{ij}$ from (1), we can extend the vectors $\Bg_i^{[j_1,\dots,j_r]}$ to the corresponding vector $\Bg_i^{[i_1,\dots,i_m]}$ in $\R^m$; these new vectors in $\R^m$ will serve as the generators of $Q_m$.

Let $\col^{[i_1,\dots,i_m]}_i$ be any column vector, where $i\in \bar{I}_{i_1,\dots,i_m}$. \ By (2),

$$\col_i^{[j_1,\dots,j_r]} =k_{i1} \Bg_{1}^{[j_1,\dots,j_r]} + \dots + k_{ir} \Bg_{r}^{[j_1,\dots,j_r]}.$$

By the definition of extension, we also have

$$\col_i^{[i_1,\dots,i_m]} =k_{i1} \Bg_{1}^{[i_1,\dots,i_m]} + \dots + k_{ir} \Bg_{r}^{[i_1,\dots,i_m]}.$$
\end{proof}
It is clear that $Q_m$ has the same rank and size as $Q_r$, and thus $r=\rank(Q_m)\le r_0$ and

$$|Q_m|=O(2^d n^C/n^{r\ep/2}).$$
\end{proof}

We now deduce a useful corollary of Lemma \ref{lemma:mrows}. \ Let $H$ be the subspace obtained from (ii) of Lemma \ref{lemma:mrows}, and let $\row_{l_1}^{[\bar{I}_{i_1,\dots,i_m}]} ,\dots, \row_{l_r}^{[\bar{I}_{i_1,\dots,i_m}]}, l_1,\dots, l_r \in \{i_1,\dots,i_m\}, r=\dim(H)$ be any vectors that span $H$; we will refer to them as {\it base vectors}. \ By definition, we have the following.

\begin{claim}\label{claim:H:representation} For all $\row_{i}, i\in \{i_1,\dots,i_m\}$, the following holds: there exist real numbers $t_{i1},\dots, t_{ir}$ such that

\begin{equation}\label{eqn:representation}
\row_i^{[\bar{I}_{i_1,\dots,i_m}]} = \sum_{j=1}^r t_{ij} \row_{l_j}^{[\bar{I}_{i_1,\dots,i_m}]}.
\end{equation}
\end{claim}

One can also rewrite \eqref{eqn:representation} in a simple matrix form. \ Let $M_{i_1,\dots,i_m}$ be the invertible matrix of order $n$ obtained from $I_n$ by replacing its $i$-th rows, with $i\in \{i_1,\dots,i_m\}$, by the vector

$$(0,\dots, 0,-t_{i1},0,\dots, 0, \dots,-t_{ir},0,\dots,0,1,0,\dots, 0).$$

In other words, the matrix $M_{i_1,\dots,i_m}$ acts on $M$ by fixing every row except those $\row_i$ with $i\in \{i_1,\dots,i_m\}$ in which case

$$M_{i_1,\dots,i_m}:\row_i(M) \rightarrow \row_i(M)-\sum_{j=1}^r t_{ij} \row_{l_j}(M).$$

\begin{corollary}\label{cor:mrows} With the definition of $M_{i_1,\dots,i_m}$ as above, for any $i\in \{i_1,\dots,i_m\}$, the projections of the row vectors of the product matrix $M_{i_1,\dots,i_m} M$ onto their components of indices in $\bar{I}_{i_1,\dots,i_n}$ vanish:

$$\row_i^{[\bar{I}_{i_1,\dots,i_m}]}(M_{i_1,\dots,i_m} M)=\mathbf{0}, \forall i\in \{i_1,\dots,i_m\}.$$

\end{corollary}

\section{Treatment for general matrices: proof of Theorem \ref{theorem:discrete:main:1}}\label{section:general}

\subsection{A proof without additive structure} We first show an easier variant of Theorem \ref{theorem:discrete:main:1} where the additive structure is omitted.

\begin{theorem}\label{theorem:discrete:nostructure:1}  Let $0<\ep<1$ and $C$ be positive constants. \ Let $M=(m_{ij})_{1\le i,j\le n}\in \M(n)$ be a matrix  with $s_0(M)\ge n^{-C}$. \ Then there exists a submatrix $M'$ of $M$ of size $n_1\times n_2$, with $n_1,n_2=n-O(n^{1-\ep})$, and a set of $r=O(1)$ vectors $\Bg_1,\dots,\Bg_r\in \R^{n_2}$ such that $M'$ can be written as $M''+F$, where $F\in \CF_{n_1n_2}$ and the rows of $M''$ are generated by $\Bg_1,\dots,\Bg_r$.
 \end{theorem}

 As the subspace generated by the columns of $M''$ also has dimension at most $r$, we can restate Theorem \ref{theorem:discrete:nostructure:1} as follows.

\begin{theorem}\label{theorem:discrete:nostructure:2} There exist a submatrix $M'$ of $M$ of size $n_1\times n_2$, with $n_1,n_2=n-O(n^{1-\ep})$, and a set of $r=O(1)$ vectors $\Bg_1,\dots,\Bg_r\in \R^{n_1}$ such that  such that $M'$ can be written as $M''+F$, where $F\in \CF_{n_1n_2}$ and the columns of $M''$ are generated by $\Bg_1,\dots,\Bg_r$.
\end{theorem}

We now prove Theorem \ref{theorem:discrete:nostructure:1}. \ As the property $\P_{\Bx\in \CC_n} (M\Bx\in \CC_n)\ge n^{-C}$ does not change by swapping the rows and columns of $M$, we will apply these swaps whenever necessary to simplify the presentation.

We will apply Lemma \ref{lemma:mrows} and Corollary \ref{cor:mrows} to blocks of $m=r_0+n^{2\ep}$ consecutive rows of $M$. \ In each block $B_i$, we will keep track of the base vectors (by adding at most $r_0-1$ vectors if needed, we always assume that there are exactly  $r_0$ base vectors) and $n^{2\ep}$ others that belong to the subspace generated by these vectors. \ By swapping the rows if necessary, we obtain the following.

\begin{claim}
The row set of $M$ (with an exception of at most $r_0+n^{2\ep}-1$ last rows) can be decomposed into $n_0$ consecutive blocks of size $r_0+n^{2\ep}$ each, here $n_0\ge n/(r_0+n^{2\ep})-1$, such that in each block, the first $r_0$ rows serve as the base vectors and the next $n^{2\ep}$ rows belong to the subspace generated by these $r_0$ base vectors.
\end{claim}

As such, for each block matrix $B_i$ (generated by the rows in the $i$-th block), the matrix $M_i$ obtained in Corollary \ref{cor:mrows} is a lower triangular matrix with 1's on its main diagonal of the form
\[ M_i=
\left(\begin{array}{ccc}
I_{(i-1)(r_0+n^{2\ep})} & \mathbf{0} & \mathbf{0} \\
\mathbf{0} & T_i & \mathbf{0} \\
\mathbf{0} & \mathbf{0} & I_{n-i(r_0+n^{2\ep})}
\end{array}
\right),\]

where $T_i$ is the corresponding lower triangular square matrix of order $r_0+n^{2\ep}$.

Observe also that for any $\By=(y_1,\dots,y_n)$, $M_i \By$ just changes the components $y_{i_0}$ with  $i_0\in \{(i-1)(r_0+n^{2\ep})+r_0+1, \dots, i(r_0+n^{2\ep}-1)\}$,

\begin{equation}\label{eqn:ii_0}
(M_i\By)_{i_0}= y_{i_0} - \sum_{j=1}^{r_0} t_{i_0j} y_{(i-1)n^{2\ep} +j}.
\end{equation}

Set

\begin{equation}\label{eqn:M_0M_0'}
M_0:=M_{n_0} \cdots  M_1 \mbox{ and } M_0':=M_0  M.
\end{equation}

Let $\Bx\in \R^n$ and $\By=M\Bx$. By definition, the components of $\By\in \R^n$ can be decomposed into blocks of $r+n^{2\ep}$ consecutive components $\By_1=(y_1,\dots,y_{n^{2\ep}}), \By_2=(y_{n^{2\ep}+1},\dots, y_{2n^{2\ep}}),\dots$ (except at most $n^{2\ep} +r_0-1$ last components) such that

$$M_0 \By =(T_1\By_1,T_2 \By_2, \dots).$$

Also, by definition,

\[ M_0'=
\left(\begin{array}{c}
T_1 B_1 \\
T_2 B_2 \\
\cdots
\end{array}
\right),\]

where we recall that $B_i$ is the $(r_0+n^{2\ep})\times n$ matrix generated by $\row_j(M), (i-1)(r_0+n^{2\ep})+1 \le j \le i(r_0+n^{2\ep})$. \ Furthermore, by Corollary \ref{cor:mrows}

\[ T_i B_i=
\left(\begin{array}{cc}
X_i \\
Y_i
\end{array}
\right),\]

where $X_i$ is an $r_0\times n$ matrix and $Y_i$ is an $n^{2\ep} \times n$ matrix satisfying the following property: there exists an exceptional index set $I_i\subset [n]$ of size at most $(k+1)n^\ep < n^{3\ep}$ such that for any $j\in  \bar{I}_i$,

\begin{equation}\label{eqn:Y:sparse}
\col_j(Y_i) = \mathbf{0}.
\end{equation}

Now we analyze the event $\By=M\Bx\in \CC_n$. \ Rewrite as

\begin{equation}\label{eqn:M'M_0}
M_0'\Bx=M_0\By.
\end{equation}

Projecting \eqref{eqn:M'M_0} onto the components of indices from $\{(i-1)(r_0+n^{2\ep})+r_0+1, \dots, i(r_0+n^{2\ep})\}$, we obtain via \eqref{eqn:ii_0}

\begin{equation}\label{eqn:M'M_0:i}
Y_i \Bx =  \Big(y_{i_0} - \sum_{j=1}^{r_0} t_{i_0j} y_{(i-1)(r_0+n^{2\ep}) +j} \Big)_{i_0\in \{(i-1)(r_0+n^{2\ep})+r_0+1, \dots, i(r_0+n^{2\ep})\}}.
\end{equation}

Furthermore, when $\By \in \CC_n$,

$$(y_{i_0})_{i_0\in \{(i-1)(r_0+n^{2\ep})+r_0+1, \dots, i(r_0+n^{2\ep})\}}\in \CC_{n^{2\ep}} \mbox{ and } (y_{(i-1)(r_0+n^{2\ep}) +1}, \dots,  y_{(i-1)(r_0+n^{2\ep}) +r_0}) \in \CC_{r_0}.$$

Define $L_i$ to be the $n^{2\ep} \times r_0$ deterministic matrix

$$L_i\Bz:=(- \sum_{j=1}^{r_0} t_{i_0j} z_j)_{i_0\in \{(i-1)(r_0+n^{2\ep})+r_0+1, \dots, i(r_0+n^{2\ep})\}}.$$

We obtain from \eqref{eqn:M'M_0:i}  the following useful bound.

\begin{lemma}[Block structure I]\label{lemma:blocks:str1}
Assume that $\P_{\Bx\in \CC_n} (M\Bx\in \CC_n) \ge n^{-C}$. Then,
$$\P_{\Bx\in \CC_n} ( \wedge_{1\le i\le n_0}Y_i \Bx \in L_i\CC_{r_0}+\CC_{n^{2\ep}}) \ge n^{-C}.$$

\end{lemma}

{\bf Independence decomposition}. As a consequence of Lemma \ref{lemma:blocks:str1}, for any $j_0\le n_0$ and any $j_0$ indices $i_1<\dots<i_{j_0}$, one also has

\begin{equation}\label{eqn:independence:1}
\P_{\Bx\in \CC_n} ( \wedge_{1\le j\le j_0}Y_{i_j} \Bx \in L_{i_j}\CC_{r_0}+\CC_{n^{2\ep}}) \ge n^{-C}.
\end{equation}

Recall from \eqref{eqn:Y:sparse} that the number of non-zero columns of each $Y_{i_j}$ is $|I_{i_j}|\le n^{3\ep}$, and thus they are extremely sparse. \ Furthermore, if the index sets $I_{i_1},\dots I_{i_{j_0}}$ were disjoint for some $j_0\gg \log n$, then the events $Y_{i_j} \Bx \in L_{i_j}\CC_{r_0}+\CC_{n^{2\ep}}$ would be independent, and so
\eqref{eqn:independence:1} would imply that most of the events $Y_{i_j} \Bx \in L_{i_j}\CC_{r_0}+\CC_{n^{2\ep}}$ hold with probability very close to one. \ We will make this observation rigorous in the next step.

Let $Y_{i_j | i_1,\dots, i_{j-1}}$  be the (possibly empty) submatrix of $Y_{i_j}$ of columns indexing from $I_{i_j}\backslash (I_{i_1}\cup \dots I_{i_{j-1}})$, we can rewrite \eqref{eqn:independence:1} in terms of conditional probability

\begin{align}\label{eqn:prob:conditional}
n^{-C} &\le \P_{\Bx\in \CC_n} ( \wedge_{1\le j\le j_0}Y_{i_j} \Bx \in L_{i_j}\CC_{r_0}+\CC_{n^{2\ep}})  \nonumber \\
&= \P_{\Bx_{I_{i_1}}} ( Y_{i_1} \Bx_{I_{i_1}} \in L_{i_1}\CC_{r_0}+\CC_{n^{2\ep}})  \times \P_{\Bx_{I_{i_2}}} ( Y_{i_2} \Bx_{I_{i_2}} \in L_{i_2}\CC_{r_0}+\CC_{n^{2\ep}}| A(\Bx_{I_{i_1}})) \times  \dots  \nonumber \\
&\times  \P_{\Bx_{I_{i_{j_0}}}} ( Y_{i_{j_0}} \Bx_{I_{i_{j_0}}} \in L_{i_{j_0}}\CC_{r_0}+\CC_{n^{2\ep}}| A(\Bx_{I_{i_1}}) \wedge \dots \wedge A(\Bx_{I_{i_{j_0-1}}})) \nonumber \\
&\le \P_{\Bx_{I_{i_1}}\in \CC_{|I_{i_1}|}} ( Y_{i_1} \Bx_{I_{i_1}} \in L_{i_1}\CC_{r_0}+\CC_{n^{2\ep}}) \times  \sup_{\Ba_2}\P_{\Bx_{I_{i_2} \backslash I_{i_1}}\in \CC_{|I_2/I_1|}} ( Y_{i_2|i_1} \Bx_{I_{i_2}\backslash I_{i_1}} \in \Ba_2+   L_{i_2}\CC_{r_0}+\CC_{n^{2\ep}})   \times  \dots  \nonumber \\
&\times \sup_{\Ba_{j_0}}\P_{\Bx_{I_{j_0}\backslash I_{i_1}\cup \dots \cup I_{i_{j_0-1}}} \in \CC_{|I_{i_{j_0}}\backslash I_{i_1}\cup \dots \cup I_{i_{j_0-1}}|}} ( Y_{i_{j_0}|i_1,\dots, i_{j_0-1}} \Bx_{I_{i_{j_0}}\backslash I_{i_1}\cup \dots \cup I_{i_{j_0-1}} } \in  \Ba_{j_0}+ L_{i_{j_0}}\CC_{r_0}+\CC_{n^{2\ep}}),
\end{align}

where $A(\Bx_{I_{i_j}})$ are the events $Y_{i_{j}} \Bx_{I_{i_{j}}} \in L_{i_{j}}\CC_{r_0}+\CC_{n^{2\ep}}$.

\begin{definition}
Motivated by \eqref{eqn:prob:conditional}, we say that a subsequence $\{Y_{i_1},\dots, Y_{i_l}\}$ is {\it bad} if for any $1\le k\le l$,

$$\sup_{\Ba}\P_{\Bx_{I_{i_k}\backslash \cup_{j=1}^{k-1}I_{i_j}}  \in \CC_{|I_{i_k}\backslash \cup_{j=1}^{k-1}I_{i_j}|} } ( Y_{i_k|i_1,\dots, i_{k-1}} {\Bx_{I_{i_k}\backslash \cup_{j=1}^{k-1}I_{i_j}}} \in  \Ba+ L_{i_k}\CC_{r_0}+\CC_{n^{2\ep}}) \le 1-\ep.$$

Note that this definition trivially implies

$$|I_{i_k}\backslash (I_{i_1}\cup I_{i_2}\cup \dots \cup I_{i_{k-1}})| \ge 1.$$

\end{definition}

\begin{claim}
If the subsequence $\{Y_{i_1},\dots, Y_{i_l}\}$ is bad, then

$$l \le (2C/\ep) \log n.$$
\end{claim}
\begin{proof} This follows directly from \eqref{eqn:prob:conditional},

$$n^{-C}\le (1-\ep)^l.$$
\end{proof}

In our next step, choose a longest possible bad sequence. \ Without loss of generality, we assume that this consists of $Y_1,\dots, Y_l$. \ Next, for $l+1\le i\le n_0$, call $i$ {\it suitable} if

$$|J_i|:=|I_i\backslash I_1\cup \dots \cup I_l| \ge 1.$$

Let $i\in \CI_g$ be a suitable index, we will be focusing on the structure of the matrix $Y_{i|i_1,\dots,i_l}$.

\begin{lemma}[Block structure II]\label{lemma:inverse:local}
Let $A$ be a matrix of size $n^{2\ep} \times k$, where $k\ge 1$ and

$$\sup_{\Ba} \P_{\Bx\in \CC_k}(A \Bx \in \Ba +  L\CC_{r_0}+\CC_{n^{2\ep}}) \ge 1-\ep,$$

for some deterministic $n^{2\ep} \times r_0$ matrix $L$. \ Then there exists an index set $I_A \subset [k]$ of size $O_{r_0}(1)$ such that the submatrix of $A$ generated by columns indexing from $\bar{I}(A)$ has the following property: every row vector is either zero or contains exactly one $\pm 1$ entry.
\end{lemma}

In our analysis $A$ will play the role of the matrices $Y_{i|i_1,\dots,i_l}$.

\begin{proof}(of Lemma \ref{lemma:inverse:local}) The assumption $\sup_{\Ba} \P_{\Bx\in \CC_k}(A \Bx \in \Ba +  L\CC_{r_0}+\CC_{n^{2\ep}}) \ge 1-\ep$ implies

\begin{equation}\label{eqn:local:large}
\sup_{\Ba} \P_{\Bx\in \CC_k}(A \Bx \in \Ba +\CC_{n^{2\ep}}) \ge (1-\ep)/2^{r_0}.
\end{equation}

We apply Theorem \ref{theorem:optimal} (more precisely Erd\H{o}s' bound for the forward Littlewood-Offord problem, \cite{E}) for any row $\row_i=(a_{i1},\dots,a_{ik})$ of $A$. \ It is implied that all but $O_{r_0}(1)$ components $a_{ij}$ are zero. \ For each $1\le i\le n^{2\ep}$, let $E_i\subset [n]$ denote the index set of the non-zero elements $a_{ij}, 1\le j\le k$. \ Similarly to what we have done so far, call a sequence $E_{i_1},\dots, E_{i_s}$ {\it ill} if one of the following holds:

\begin{itemize}
\item either $|E_{i_j}\backslash E_{i_1}\cup \dots \cup E_{i_{j-1}}|\ge 2$,
\vskip .1in
\item or $|E_{i_j}\backslash E_{i_1}\cup \dots \cup E_{i_{j-1}}|=1$, and the corresponding unique non-zero element $a_{i_j*}$ is different from $\pm1$.
\end{itemize}

If all of the $E_i$ have size at most 1 and the corresponding non-zero elements (if any) are either $1$ or $-1$, then we are done. \ Otherwise, choose a longest possible ill subsequence, and without loss of generality, assume that this sequence consists of $E_1,\dots, E_s$.

\begin{claim}\label{claim:s:r_0}
One has

$$s\le r_0.$$
\end{claim}

\begin{proof}(of Claim \ref{claim:s:r_0}) Observe that if the sequence $a_1,\dots,a_k\in \R$ has at least two non-zero entries or one non-zero entry different from $\pm 1$, then $\P(\sum_i \ep_i a_i \in a+\{-1,1\}) \le 1/2$. \ To complete the proof one just needs to rewrite \eqref{eqn:local:large} as product of conditional probabilities as in \eqref{eqn:prob:conditional},

$$(1-\ep)/2^{r_0} \le (1/2)^s.$$
\end{proof}

We now use Claim \ref{claim:s:r_0} to complete the proof of Lemma \ref{lemma:inverse:local}. Let $s+1\le i\le n^{2\ep}$ be an arbitrary index. \ Then, as $E_1,\dots, E_s$ is longest possible, $|E_i \backslash E_1 \cup \dots \cup E_s|\le 1$, and if equality holds then the corresponding non-zero element $a_{i*}$ must be either $1$ or $-1$. \ Set

$$I(A):= \cup_{1\le i\le s} E_i.$$

Then $|I(A)| \le s O_{r_0}(1) = O_{r_0}(1)$, completing the proof.
\end{proof}

To proceed further, we apply Lemma \ref{lemma:inverse:local} to each $A=Y_{i|i_1,\dots,i_l}, i\in \CI_g$, and set

$$I(M_0'):=I_1\cup \dots \cup I_l \cup_{i\in \CI_g} I(A).$$

It is clear that this index set has size at most

$$|I(M_0')| \le O(n^{3\ep}\log n)+ n_0 O_{r_0}(1) = O(n^{1-2\ep}).$$

Putting together, we have obtained the following: for every vector $\row_{i_0}(M_0')$ with $i_0\in \{(i-1)(r_0+n^{2\ep})+r+1,\dots, i(r_0+n^{2\ep}), 1\le i\le n_0\}$, its restriction over the components indexing from $\bar{I}(M_0')$, $\row_{i_0}^{[\bar{I}(M_0')]}$, is either zero or contains exactly one element from $\pm1$. \ Recall the definition of $M_0'$ from \eqref{eqn:M_0M_0'}, we can restate the result in terms of $M$ as follows.

\begin{lemma}\label{lemma:summary}
There exists a set $I(M)(=I(M_0'))$ of exceptional indices with $|I(M)|=O(n^{1-2\ep})$ such that the following holds for each $\row_{i_0}(M)$ with $i_0\in \{(i-1)(r_0+n^{2\ep})+r_0+1,\dots, i(r_0+n^{2\ep}), 1\le i\le n_0\}$: the row vectors $\row_{i_0}^{[\bar{I}(M)]}- \sum_{k=1}^{r_0} t_{i_0k}\row_{(i-1)(r_0+n^{2\ep})+k}^{[\bar{I}(M)]}$ are either zero or contain exactly one element from $\pm1$.
\end{lemma}

To complete the proof of Theorem \ref{theorem:discrete:nostructure:1}, we show that there is a projection on to $n-O(n^{1-\ep})$ components for which the base vectors $\row_{(i-1)(r_0+n^{2\ep})+k}, 1\le i\le n_0,1\le k\le r_0$, generate a  subspace of dimension at most $r_0$. \ For convenience, write $\row_{ik}:= \row_{(i-1)(r_0+n^{2\ep})+k}, 1\le i\le n_0, 1\le k\le r_0$. Lemma \ref{lemma:mrows} (ii) applied to these $m$ vectors, with $m=n_0 r_0$, implies that there exist $s$ vectors $\row_{i_1},\dots, \row_{i_s}$ among the base vectors $\row_{ik}$, where $s\le r_0$, and an index set $J=J(M)$ of size $|J|=O(m n^\ep)= O(n^{1-\ep})$  such that the following holds for any $\row_{ik}$:

$$\row_{ik}^{[\bar{J}]} \in span (\row_{i_1}^{[\bar{J}]},\dots, \row_{i_s}^{[\bar{J}]}).$$

Let $N_1$ be the set of row indices from $\{(i-1)(r_0+n^{2\ep})+r_0+1,\dots, i(r_0+n^{2\ep}), 1\le i\le n_0\}$, and $N_2$ be the set of column indices from $\bar{I}(M)\cap \bar{J}(M)$. \ The proof of Theorem \ref{theorem:discrete:nostructure:1} is completed by setting $M'$ to be the submatrix of $M$ generated by the rows indexing in $N_1$ and by the columns indexing in $N_2$.

 \subsection{Proof of Theorem \ref{theorem:discrete:main:1}} We will mainly be focusing on the matrix $M'$ obtained in Theorem \ref{theorem:discrete:nostructure:1}. \ Assume that $r=\rank(M'')$, where $r\le r_0$. \ Assume also that $\row_{i_1}(M''),\dots,\row_{i_r}(M''), \\ i_1,\dots, i_r \in N_1$, span the whole row space of $M''$. \ Consider the corresponding row vectors $\row_{i_1},\dots, \row_{i_r}$ of $M$. \ Note that, by the definition of $\CF_{n_1n_2}$, when restricting $\row_{i_j}(M)$ to $M'$, the vectors $\row_{i_j}(M')$ are different from $\row_{i_j}(M'')$ in at most one component. \ Let $I(N_2) \subset N_2$ be the set of column indices of the components where $\row_{i_j}(M')$ differ from $\row_{i_j}(M''), 1\le j\le r$. \ Then $|I(N_2)|\le r$.

Lemma \ref{lemma:mrows} (i) applied to the vectors $\row_{i_1}(M),\dots, \row_{i_r}(M)$ implies a GAP $Q_r\subset \R^r$ of small size and bounded rank and a small set $I_{i_1,\dots,i_r}$ of exceptional indices such that that $Q_r$ contains all restricted columns $\col_i^{[i_1,\dots,i_r]}(M),  i \in \bar{I}_{i_1,\dots,i_n}$. \ In particular, $Q_r$ contains all columns $\col_i^{[i_1,\dots,i_r]}(M'')$, where  $i \in N_2':=\bar{I}_{i_1,\dots,i_n} \cap N_2\cap \bar{I}(N_2)$. \ By passing to a GAP of smaller rank (and still of small size) if needed, one can assume that these restricted column vectors indeed span $Q_r$ (see for instance \cite[Section 8]{TV0}), where we recall the notion of spanning from Section \ref{section:lemmas}.

To this end, because $\row_{i_1}(M''),\dots,\row_{i_r}(M'')$ span the whole row space of $M''$, we just follow the proof of Lemma \ref{lemma:mrows} (i) identically to show that $Q_r\subset \R^r$ can be extended to another GAP $Q\subset \R^{n_1}$ of the same rank and size which contains all of the columns of $M''$ of indexing in $N_2'$. \ Finally, one deletes from $M'$ the columns indexing in $\Bar{N}_2'$ to obtain the new $M'$, which clearly satisfies all of the desired properties (noting that columns deletion does not affect the property of $\CF$.)

\section{Treatment for orthogonal matrices: proof of Theorem \ref{theorem:discrete:main:2}}\label{section:orthogonal}

We first show that, for orthogonal matrices, the structures in Theorem \ref{theorem:discrete:main:1} can be extended to the whole columns without significant increase of the sets of exceptional indices.

\begin{theorem}\label{theorem:discrete:main:2'} Let $0<\ep<1$ and $C$ be positive constants. \ Let $M=(m_{ij})_{1\le i,j\le n}\in \O(n)$ be an orthogonal matrix with $s_0(M)\ge n^{-C}$ for some positive constant $C>0$. \ Then there exist a submatrix $M'$ of $M$ of size $n\times n_2$, with $n_2=n-O(n^{1-\ep})$, and a set of $r=O(1)$ vectors $\Bg_1,\dots,\Bg_r\in \R^{n}$ such that $M'$ can be written as $M''+F$, where $F\in \CF_{nn_2}$ and the columns of $M''$ belong to a GAP of small size generated by $\Bg_1,\dots,\Bg_r$.
\end{theorem}

\begin{proof}(of Theorem \ref{theorem:discrete:main:2'}) We first apply Theorem \ref{theorem:discrete:main:1}. \ Without loss of generality, assume that $M'$ consists of the first $n_1$ rows and $n_2$ columns of $M$. \ For $1\le i\le n_2$, consider the $\R^{n_1}$ column vectors of $M'$, $\col_i(M')=(m_{1i},\dots,m_{n_1i})$ and the $\R^{n-n_1}$ vectors $\Bu_i=(m_{(n_1+1)i},\dots, m_{ni})$. \ Using the definition from the proof of Theorem \ref{theorem:discrete:main:1}, one can write

$$\col_i(M')=\col_i(M'')+ \col_i(F),$$

where $r=\rank(M'')=O(1)$ and $F\in \CF_{n_1n_2}$.

Note that, as $\rank(F)\ge \rank(M')-\rank(M'') \ge n-O(n^{1-\ep})-r$, and as $F$ contain at most $n-O(n^{1-\ep})$ non-zero entries, by deleting at most $O(n^{1-\ep})$ columns of $M'$ if needed, one can assume that each column of $F$ contains at most one non-zero entry (or exactly one, but we don't need this fact).

Similarly, as $\dim(span(\col_{i_1}(M),\dots, \col_{i_k}(M)))=k$, we have

\begin{equation}\label{eqn:dim:increase}
\dim(span(\col_{i_1}(M'),\dots,\col_{i_k}(M')))\ge k-r-(n-n_1) =k-r-O(n^{1-\ep}).
\end{equation}

Roughly speaking, our next move consists of two main steps.
\begin{enumerate}[(1)]
\item Starting from $M'$, we showed that there exists an index set $T$ of size $n-O(n^{1-\ep})$ such that when restricting the low rank part $M''$ of $M'$ onto $T$,
$$\row_i^{[T]}(M) \in span\big(\row_i^{[T]}(M''), 1\le i\le n_2\big), \forall n_1+1 \le i\le n.$$
 \vskip .1in
 \item Hence, by the argument from the proof of (i) of Lemma \ref{lemma:mrows}, the GAP containing the columns $\col_i(M''), i\in T$ can be extended to a GAP that contains the column vectors $\col_i(M),i\in T$.
\end{enumerate}

We next explain these ideas in more detail.

{\bf Step 1.} Consider the following process. \ At step $0$, the row index set $S\subset [n_1]$ is set to be empty. \ At step $1\le i$, if possible, we choose $s$ column indices $i_1,\dots, i_{s}\in [n_2]$ that have never been used before, for some $s\le r+1$, so that the following holds.

\begin{enumerate}[(i)]
\item The row indices of the (possibly) non-zero entries of $\col_{i_1}(F),\dots, \col_{i_s}(F)$ must not belong to $S$.
\vskip .1in
\item There exist real coefficients $\alpha_{i_1},\dots,\alpha_{i_s}$ such that $\alpha_{i_1} \col_{i_1}(M'')+\dots + \alpha_{i_s} \col_{i_s}(M'')=\0$ but $\alpha_{i_1} \Bu_{i_1}+\dots + \alpha_{i_s} \Bu_{i_s} \neq \0$, where we recall that $\Bu_i=(m_{(n_1+1)i},\dots, m_{ni})$.
\end{enumerate}

We then add the row indices of the possibly non-zero entries of $\col_{i_1}(F),\dots, \col_{i_s}(F)$ to $S$ and move on to the next step. \ The process terminates if we are not able to proceed further.

By definition, the linear combination $\sum_{j=1}^s\alpha_{i_j}\col_{i_j}(M)$ in each step equals $\sum_{j=1}^s\alpha_{i_j}( \col_{i_j}(F) \oplus \Bu_{i_j})$. \ But as the indices $i_j$ are chosen to be disjoint from the previously-used indices $i_{j'}$, and as the columns of $M$ are orthogonal,

$$(\sum_{j=1}^s\alpha_{i_j}\col_{i_j}(M)) \cdot (\sum_{j'=1}^{s'}\alpha_{i_{j'}}\col_{i_{j'}}(M))=0.$$

On the other hand, by (i) and by definition that each column of $F$ contains at most one non-zero entry, $(\sum_{j=1}^s\alpha_{i_j}\col_{i_j}(F)) \cdot (\sum_{j'=1}^{s'}\alpha_{i_{j'}}\col_{i_{j'}}(F))=0$. \ This implies that

\begin{equation}\label{eqn:u:orthogonal}
(\sum_{j=1}^s\alpha_{i_j}\Bu_{i_j}) \cdot (\sum_{j'=1}^{s'}\alpha_{i_{j'}}\Bu_{i_{j'}})=0.
\end{equation}

By \eqref{eqn:u:orthogonal} and by the fact from (ii) that all the vectors $\sum_{j=1}^s\alpha_{i_j}\Bu_{i_j}$ are non-zero in $\R^{n-n_1}$, our process must terminate after at most $n-n_1=O(n^{1-\ep})$ steps. \ As such, the final index set $S$ has size at most

\begin{equation}\label{eqn:size:J0}
|S|\le (r+1)(n-n_1)=O(n^{1-\ep}).
\end{equation}

Now we consider the collection of columns $\col_i(M')$ of $M'$ where the row index of the possibly non-zero entries of $\col_{i}(F)$ does not belong to $S$. \ Let $T\subset [n_2]$ be the collection of these indices. \ By \eqref{eqn:dim:increase}, \eqref{eqn:size:J0}, and again by the assumption that each column of $F$ contains at most one non-zero entry,

$$|T|=n-O(n^{1-\ep}).$$

For the process cannot be continued, as by (ii), any vanishing linear combination $\alpha_{i_1} \col_{i_1}(M'')+\dots + \alpha_{i_s} \col_{i_s}(M'')=\0, i_1,\dots,i_s\in T$, also implies $\alpha_{i_1} \Bu_{i_1}+\dots + \alpha_{i_s} \Bu_{i_s} = \0$. \ In other words,

\begin{equation}\label{eqn:orthogonal:span}
\row_i^{[T]}(M) \in span\big(\row_i^{[T]}(M''), 1\le i\le n_2\big), \forall n_1+1 \le i\le n,
\end{equation}

where we recall that $\row_i^{[T]}(M)$ denote the projection of $\row_i(M)$ onto the components indexing in $T$, and similarly for $\row_i^{[T]}(M'')$.

{\bf Step 2.} It follows from \eqref{eqn:orthogonal:span}, by the same argument as in the proof of (i) of Lemma \ref{lemma:mrows} (and also of Theorem \ref{theorem:discrete:main:1}), that the GAP containing the columns of $M''$, $\col_i(M''), i\in T$, can be extended to a GAP of the same rank and size that contains the column vectors $\col_i(M)$. \ We complete the proof by letting the new $M'$ be the restriction of $M$ onto the column index set $T$ (and hence the new $F$ is obtained from the old one by adding another $n-n_1$ zero rows).
\end{proof}


To prove Theorem \ref{theorem:discrete:main:2}, we need further preparations. \ By permuting the columns if necessary, one can assume that the first $r$ columns of $M''$ (corresponding to $M'$ of size $n\times n_2$ obtained from Theorem \ref{theorem:discrete:main:2'}) span the whole columns of $M''$. \ To avoid trivial degeneracy, we will regularize the system further as follows.

\begin{claim}[Regularization]\label{claim:regularization}
With an extra loss of at most $O(n^{1-\ep})$ in the number of columns of $M'$, one can assume that every $n_2-r$ rows (in $\R^{r}$) among $n$ rows of the $n\times r$ matrix spanned by $\col_1(M''),\dots,\col_r(M'')$ have full rank.
\end{claim}

\begin{proof}(of Claim \ref{claim:regularization}) Assume otherwise that there are $n_2-r$ rows which span a subspace of dimension at most $r-1$. \ We next restrict $M'$ onto these row indices. By the assumption that the first $r$ columns of the original $M''$ span its column space, the columns of the new (after restriction) $M''$ belongs to the span of its first $r$ columns, and hence a subspace of dimension at most $r-1$.

By applying the steps (1) and (2) of the proof of Theorem \ref{theorem:discrete:main:2'}, one obtains a new $M'$ of size $n\times n_2'$, with $n_2'=n_2-r-O(n^{1-\ep})$, where the columns of the additive part $M''$ span a subspace of dimension at most $r-1$. \ In the next step, choose $r-1$ columns that span this subspace. \ We continue the process if there are $n_2'-(r-1)$ rows (in the submatrix generated by these $r-1$ columns) that span a subspace of dimension at most $r-2$ in $\R^{r-1}$, etc. \ As the process must terminate after at most $r$ steps, the number of columns remaining at termination is at least $n-r\times O(n^{1-\ep})=n-O(n^{1-\ep})$.
\end{proof}

With Claim \ref{claim:regularization} in hand, we now show that the additive part $M''$ can be described as in Theorem \ref{theorem:discrete:main:2}. \ As usual, we can simplify our matrix further as follows.

\begin{itemize}
\item By deleting the columns where $F$ vanishes or has more than one non-zero entries, one can assume that each column of $F$ in $M'=M''+F$ contains exactly one non-zero entry, which can be assumed to be $1$ after an appropriate column sign change.
\vskip .1in
\item Also, by permuting the columns, one can assume that $M'$ consists of the first $n_2$ columns of $M$, and the first $r$ columns of $M''$ span the whole column vectors of $M'$.
\vskip .1in
\item Finally, by permuting the rows, one can assume that the first $n_2$ columns of $M$ take the form $\Bu_i+\Be_i$, where $\Bu_i=(u_{1i},\dots, u_{ni}) =\col_i(M'') \in \R^n$ belongs to a GAP $Q$ of bounded rank and small size, and $\Be_i$ are the standard vectors.
\end{itemize}

We restate Theorem \ref{theorem:discrete:main:2} as follows.

\begin{theorem}\label{theorem:discrete:main:2''}
For any $r+1\le i\le n_2$, one can represent $\Bu_i$ as

$$\Bu_i = d_{i1}\Bu_1+\dots+ d_{ir}\Bu_r,$$

where $d_{i1},\dots, d_{ir}$ are uniquely determined from the first $r$ column vectors $\Bu_1,\dots,\Bu_r$ by the formula

$$(u_{i1},\dots, u_{ir}) = d_{i1} (u_{11},\dots, u_{1r}) +\dots + d_{ir} (u_{r1},\dots, u_{rr}).$$

In other words,  the matrix $M''$, when restricting to the first $n_2$ rows, can be written as

\[\left( \begin{array}{cc}
U & UD^T\\
DU  & DUD^T
\end{array}\right),
\]

where $U$ is a square matrix of size $r$, and $D$ is an $(n_2-r)\times r$ matrix.
\end{theorem}

\begin{proof}(of Theorem \ref{theorem:discrete:main:2''}) Because the first $r$ columns span the whole column space of $M''$, for any $r+1\le i\le n_2$, there exist real numbers $x_1,\dots,x_r$ (not necessarily uniquely determined) such that $\Bu_i =x_1 \Bu_1+\dots + x_r \Bu_r$. \ For any $1\le k\le r$, the condition of orthogonality $(\Bu_i+\Be_i) \cdot (\Bu_k+\Be_k)=0$ implies that

\begin{align*}
0&=\sum_{1\le j\le r} x_j (\Bu_j\cdot \Bu_k + \Bu_j \cdot \Be_k) + \Be_i \cdot \Bu_k\\
&=-\sum_{1\le j\le r} x_j \Be_j \cdot \Bu_k+ \Be_i \cdot \Bu_k = (\Be_i-\sum_{1\le j\le r} x_j \Be_j)\cdot \Bu_k,
\end{align*}

where we used the initial assumption that $(\Bu_j+\Be_j) \cdot (\Bu_k+\Be_k) = \delta_{jk}$ for $1\le j\le r$.

On the other hand, we can rewrite $(\Be_i-\sum_{1\le j\le r} x_j \Be_j)\cdot \Bu_k =0, 1\le k\le r$, as

\begin{equation}\label{eqn:orthogonal:linearcombination}
(u_{i1},\dots, u_{ir}) = x_1 (u_{11},\dots, u_{1r}) +\dots + x_r (u_{r1},\dots, u_{rr}).
\end{equation}

In particular, the first $r$ rows $(u_{11},\dots, u_{1r}),\dots,(u_{r1},\dots, u_{rr})$ span the whole row space of the $n_2\times r$ submatrix spanned by $\Bu_1^{[1,\dots,n_2]},\dots,\Bu_r^{[1,\dots,n_2]}$. \ By the assumption of Claim \ref{claim:regularization}, these $n_2$ row vectors have full rank in $\R^r$, and so the representation in \eqref{eqn:orthogonal:linearcombination} is unique: the coefficients $(x_1,\dots,x_r)$ must equal $(d_{i1},\dots,d_{ir})$ introduced in the statement.

\end{proof}

\section{Application: general matrices}\label{section:application:1}

As a first application, we deduce from Theorem \ref{theorem:discrete:main:1} that general matrices of sufficiently small entries cannot be near-invariant with respect to the hypercube.

\begin{theorem}\label{theorem:application:generic} For any $C>0$ and $0<\ep<1/2$, there exist $n_0=n_0( C,\ep)$ and $c=c( C,\ep)>0$ such that the following holds for all $n\ge n_0$. \ Let $M=(m_{ij})_{1\le i,j\le n}$ be a matrix with $\rank(M)\ge (1/2+\ep)n$ and $|m_{ij}|\le c, \forall{1\le i,j\le n}$. \ Then $s_0(M)\le n^{-C}$.

\end{theorem}

We have not tried to sharpen the requirement of $\rank(M)\ge (1/2+o(1))n$ here, but it can be easily seen that for Theorem \ref{theorem:application:generic}, the rank of $M$ must be sufficiently large. We also invite the reader to consult Appendix \ref{section:stochastic} for a related result with explicit constants for stochastic matrices of large permanent.


To prove Theorem \ref{theorem:application:generic}, we need a simple claim stated below.

\begin{claim}\label{claim:application:identity}
Assume that $F\in \CF_{n_1n_2}$ with $n_1,n_2=(1-o(1))n$ and $\rank(F)\ge (1/2+\ep)n$. \ Then $F$ contains a block of size $\ep n \times \ep n$  which contains exactly one non-zero entry in each row and column.
\end{claim}

\begin{proof}(of Claim \ref{claim:application:identity}) Recall that each row of $F$ is either zero or contains exactly one non-zero entry. \ Thus the total number of non-zero entries of $F$ is at most $n_1$. \ As such, the number of columns of $F$ that contain at least two non-zero entries is at most $n_1/2$. \ Consider the submatrix $F'$ of $F$ obtained by deleting these columns; then

$$\rank(F')\ge \rank(F)-n_1/2 \ge \ep n.$$

By definition, each row and column of $F'$ has at most one non-zero entry; thus $F'$ contains a block of size $\ep n \times \ep n$ which contains exactly one non-zero entry in each row and column as desired.
\end{proof}

\begin{proof}(of Theorem \ref{theorem:application:generic}) Assume otherwise. \ After appropriate row and column permutations and sign changes, by Theorem \ref{theorem:discrete:main:1} and Claim \ref{claim:application:identity}, one can assume that the top-left $n'\times n'$ corner of $M$, where $n'=\ep n$, can be written as $U+F$, with $F$ being the identity matrix $I_{n'}$ and the column vectors $\Bu_1,\dots,\Bu_{n'}$ of $U$ belonging to a GAP $P$ of small size and bounded rank $r$. \ In what follows we will not use this information in its full strength, but only a weaker fact, that the space generated by $\Bu_i$ has bounded dimension $r$ (hence, Theorem \ref{theorem:discrete:nostructure:1} would suffice). \ Without loss of generality, assume that the first $r$ columns $\Bu_1,\dots, \Bu_r$ span this subspace in $\R^{n'}$.

For any $r+1\le i\le n'$, there exist $c_{i1},\dots, c_{ir}$ such that $\Bu_i = c_{i1}\Bu_1+\dots+c_{ir}\Bu_r$. \ In particular, for $i=r+1$,

\begin{equation}\label{eqn:application:orthogonal:1}
\Bu_{r+1} = c_{(r+1)1}\Bu_1+\dots+c_{(r+1)r}\Bu_r.
\end{equation}

Thus, in the $(r+1)^{th}$ component,

\begin{equation}\label{eqn:application:orthogonal:1'}
u_{(r+1)(r+1)} = c_{(r+1)1}u_{(r+1)1}+\dots+c_{(r+1)r}u_{(r+1)r}.
\end{equation}

By definition, as $|u_{(r+1)(r+1)}+1|=|m_{(r+1)(r+1)}|\le c$,

\begin{equation}\label{eqn:application:orthogonal:2}
|u_{(r+1)(r+1)}| \ge 1-c.
\end{equation}

As the off-diagonal terms, $u_{(r+1)1}=m_{(r+1)1}, \dots, u_{((r+1)r}=m_{(r+1)r}$, all have absolute value at most $c$ by assumption, we have

$$1-c \le c(|c_{(r+1)1}| + \dots + |c_{(r+1)r}|) \le rc \max \{|c_{(r+1)1}|,\dots,|c_{(r+1)r}|\}.$$

Assume that the maximum above is achieved at some $1\le i_0\le r$,

\begin{align}\label{eqn:application:orthogonal:2'}
|c_{(r+1)i_0}|=\max \{|c_{(r+1)1}|,\dots,|c_{(r+1)r}|\} \ge (1-c)/rc.
\end{align}

Next, consider \eqref{eqn:application:orthogonal:1} in the $i_0^{th}$ component,

$$u_{i_0(r+1)} = c_{((r+1)1}u_{i_01}+\dots+c_{(r+1)r}u_{i_0r}.$$

Equivalently,

\begin{align}\label{eqn:application:orthogonal:3}
-c_{(r+1)i_0}u_{i_0i_0}= &-u_{i_0(r+1)} + c_{((r+1)1}u_{i_01}+\dots+c_{(r+1)(i_0-1)} u_{i_0(i_0-1)}\nonumber \\
&+ c_{(r+1)(i_0+1)} u_{i_0(i_0+1)}+ \dots +c_{(r+1)r}u_{i_0r}.
\end{align}

On the other hand, for the same reason as in \eqref{eqn:application:orthogonal:2}, the diagonal term $u_{i_0i_0}$ has large absolute value, $|u_{i_0i_0}| \ge 1-c$. \ Similarly, the off-diagonal terms, $u_{i_01}=m_{i_01}, \dots, u_{i_0r}=m_{i_0r}, u_{i_0(r+1)}=m_{i_0(r+1)},$ all have absolute value at most $c$. \ It thus follows from \eqref{eqn:application:orthogonal:3} that

$$|c_{(r+1)i_0}|(1-c) \le c+ r c |c_{(r+1)i_0}|.$$

If $c$ is chosen to be strictly smaller $1/(1+r)$, then this gives

$$|c_{(r+1)i_0}|\le c/(1-(r+1)c).$$

However, this contradicts \eqref{eqn:application:orthogonal:2'} when $c$ is selected sufficiently small depending on $r$.
\end{proof}

\section{Application: proof of Theorem \ref{theorem:orthogonal:2}}\label{section:application:2}

In this section we prove Theorem \ref{theorem:orthogonal:2} by invoking Theorem \ref{theorem:discrete:main:2}. \ Assume that the matrix of the first $r$ columns of $M''$ has the form $(U,DU)$, where $U$ is a non-singular square matrix of size $r$, and by Claim \ref{claim:regularization}, one can also assume that $\rank(D)=r$.

The fact that the first $r$ columns are orthogonal yields the following

\begin{equation}\label{eqn:application:domination:0}
(U+I_r)^T (U+I_r) +U^T D^T D U=I_r, \mbox{ or equivalently, } U+U^T + U^T U + U^T D^T D U=0.
\end{equation}

Thus $U+U^T$ is a negative semidefinite matrix of real entries, and the diagonal terms of $DUD^T$ are non-positive because

\begin{align}\label{eqn:application:domination:1}
-(DUD^T)_{ii}=-(DU^T D^T)_{ii}&=-(D(U+U^T)D^T)_{ii}/2\nonumber \\
&=[D(U^T U +U^T D^T D U) D^T]_{ii}/2,
\end{align}

The latter is a sum of two positive semidefinite matrices of real entries.

Furthermore, \eqref{eqn:application:domination:0} also implies that $U^{-1} + (U^T)^{-1}+I_r+D^T D=0$. \ Let $A:=U^{-1}-(U^T)^{-1}$ be the matrix difference. \ Then $A$ is real asymmetric and $U^{-1}= [A- (I_r+D^T D)]/2$. \ Thus

$$U=-2(I_r+D^T D -A)^{-1}.$$

Using this formula for $U$, we show the following key trace-estimate.

\begin{lemma}\label{lemma:application:domination} We have
$$|\tr(DUD^T)| \le 2r.$$
\end{lemma}

\begin{proof}(of Lemma \ref{lemma:application:domination}) Using the formula $\tr(XY) = \tr(YX)$, write

$$\tr(DUD^T) = \tr(D^TD U)= -2\tr(D^T D (I_r+D^T D -A)^{-1}).$$

Let $V\in \O( r)$ be an orthogonal matrix which diagonalizes $D^T D$,

$$V (D^T D) V^T =E,$$

where $E=(\lambda_i)_{1\le i\le r}$ is the diagonal matrix of eigenvalues of $D^T D$ (where we recall that as $\rank(D)=r$, all the eigenvalues $\lambda_i $ of $D^T D$ are positive).

Plugging into the trace identity above,

\begin{align*}
\tr(D^T D (I_r+D^T D -A)^{-1}) &= \tr((V^T EV) (I_r+D^T D -A)^{-1}) =  \tr(EV (I_r+D^T D -A)^{-1} V^T)\\
&=\tr(E(I_r+VD^T DV^T -VAV^T)^{-1}) = \tr(E(I_r+E -VAV^T)^{-1})\\
&=\tr(E(I_r+E -B)^{-1}),
\end{align*}

where $B:=VAV^T$ is a real asymmetric matrix of order $r$.

Next, let $F$ be the diagonal matrix of positive entries $F=(\sqrt{\lambda_i})_{1\le i\le r}$. \ Thus $E=FF^T$, and so

 \begin{align*}
\tr(D^T D (I_r+D^T D -A)^{-1})&=\tr(E(I_r+E -B)^{-1})= \tr(FF^T(I_r+E -B)^{-1})\\
&=\tr(F^T(I_r+E -B)^{-1}F)=\tr(I_r+E^{-1}-F^{-1}B(F^T)^{-1})^{-1}.
\end{align*}

Again, notice that the matrix $B':=F^{-1}B(F^T)^{-1}$ is another real asymmetric matrix, and the diagonal terms of the inverse matrix $E'=E^{-1}$ are positive. \ To this end, we introduce the following estimate.

\begin{claim}\label{claim:application:trace}
Assume that $E'=(e_i)$ is a diagonal matrix with positive entries, and $B'=(b_{ij})$ is a real asymmetric matrix, all of order $r$. \ Then

$$0\le \tr(I_r+E'-B')^{-1} \le r.$$
\end{claim}

\begin{proof}(of Claim \ref{claim:application:trace})
Let $\mu_i, 1\le i\le r,$ be the eigenvalues of $I_r+E'-B'$. \ Then

$$\tr(I_r+E'-B')^{-1} = \sum_i \mu_i^{-1} = \sum_i \Re \mu_i^{-1}.$$

On the other hand, for any eigenvalue $\mu$ with unit eigenvector $\Bx=(x_1,\dots,x_r)\in \C^r$, the identity $(I_r+E'-B')\Bx =\mu \Bx$ implies that

\begin{align*}
\mu &= (I_r+E'-B')\Bx \cdot \bar{\Bx}=(I_r+E') \Bx\cdot \bar{\Bx} - B' \Bx \cdot \bar{\Bx} = 1 +\sum_i e_i |x_i|^2  - \sum_{i\neq j} b_{ij} x_i \bar{x}_j\\
&= 1 +\sum_i e_i |x_i|^2  -\sum_{1\le i<j\le r} b_{ij}(x_i \bar{x}_j - x_j \bar{x}_i).
\end{align*}

Because $b_{ij}\in \R$, the second summand is purely imaginary, and so $\Re \mu  =1 +\sum_i e_i |x_i|^2\ge 1 $. \ It thus follows that

$$0\le \Re \mu^{-1} \le 1.$$

Summing over all $\mu_i$, we hence obtain the claim.

\end{proof}

It is clear that Claim \ref{claim:application:trace} implies Lemma \ref{lemma:application:domination}. \ The proof of this lemma is therefore complete.
\end{proof}

We now conclude the main result of this section.

\begin{proof}(of Theorem \ref{theorem:orthogonal:2}) Let $k=n^{\ep}$, and let $K$ be the number of indices $i_0, 1\le i_0\le n_2-r,$ where $(DUD^T)_{i_0i_0} \le -k/n$. By \eqref{eqn:application:domination:1} and Lemma \ref{lemma:application:domination},

$$K\le \frac{2rn}{k} = 2rn^{1-\ep}.$$

Thus there are at least $n_2-r-O(n^{1-\ep})=n-O(n^{1-\ep})$ indices $i_0$ such that

$$-k/n \le (DUD^T)_{i_0i_0} \le 0.$$

With such $i_0$, by Theorem \ref{theorem:discrete:main:2},

$$|(M''+F)_{i_0i_0} |=|(DUD^T)_{i_0i_0}+ (I)_{i_0i_0}|\ge 1 -n^{-1+\ep}.$$


\end{proof}


\section{Open problems}\label{section:problem}

One obvious problem is to improve our result, to show that if $M$ maps many
points in $\mathcal{C}_{n}$\ to points in $\mathcal{C}_{n}$, then $M$\ is
close to a permuted diagonal matrix on all but $O\left(  \log n\right)
$\ rows (rather than all but $O\left(  n^{1-\varepsilon}\right)  $\ rows, as
the current result gives).

A second problem is to generalize our result (regarding $s_0(M)$ \footnote{In the complex setting, the exact score function $s_0(M)$ is the probability that $M\Bx$ lies exactly on the product of $n$ unit circles, that is $s_0(M)= \P_{\Bx\in \CC_n}(|(M\Bx)_1|=\dots = |(M\Bx)_n|=1)$.}) from orthogonal matrices and real matrices to
unitary matrices and general complex matrices.


A third problem is to prove analogous results constraining the form of
near-isometries, but for objects other than hypercubes.

However, perhaps the most interesting problem is to generalize this paper's
treatment from the \textquotedblleft exact\textquotedblright\ score function
$s_{0}\left(  M\right)  $\ to the original score function $s\left(  M\right)
$, thereby answering questions \ref{question:score} and \ref{question:perm} by the second named author and Hance
(at least for the case of real and orthogonal matrices). \ To that end, we believe that the following asymptotic version of Theorem \ref{theorem:optimal} would be useful.

\begin{conjecture}
\label{ballconj}Let $0<\ep <1$. Let $B$\ be
a ball in $\R^{d}$\ of radius $1/\sqrt{n}$ (where $d$ could be as large as $\log n$),\ and suppose that the $\R^d$ vectors $\mathbf{a}_{1},\ldots,\mathbf{a}_{n}$\ satisfy

\[
\Pr_{x_{1},\ldots,x_{n}\in\left\{  -1,1\right\}  }\left( x_{1}\mathbf{a}%
_{1}+\cdots+x_{n}\mathbf{a}_{n}\in B\right)  \geq\frac{1}{n^{C}}%
\]
for some constant $C$, where $x_{1},\ldots,x_{n}$\ are independent Bernoulli
variables. \ Then there exists a subspace $S\leq\mathbb{R}^{d}$ of dimension
$O_{C,\ep}(1)$\ (depending on $C,\ep$ but independent of $n$ and $d$), such that all but at most $n^\ep$ of the vectors $\mathbf{a}_{i}$\ have
distance at most $1/n^{\varepsilon}$\ from some vector in $S$.
\end{conjecture}


{\bf Acknowledgments.} The authors are grateful to T.~Tao, A.~Arkhipov, and S.~Garg for invaluable comments and suggestions.

\appendix

\section{Permaments of stochastic matrices}\label{section:stochastic}

The goal of this section is to show a strong variant of Theorem \ref{theorem:orthogonal:2} and Theorem \ref{theorem:application:generic} for (column) stochastic matrices of large permanent.

\begin{theorem}\label{theorem:stochastic}
Let $A=\left(  a_{ij}\right)  $\ be an $n\times n$\ stochastic matrix, and
suppose $\per(A)  \geq n^{-C}$.
\ Then all but $O_C\left( \log n\right)  $\ of the rows of $A$\ contain an
entry that is at least $0.8$, with the remaining entries in that row summing
to at most $0.1$.\ \ (Of course, by stochasticity, these $0.8$\ entries must
all lie in separate columns.)
\end{theorem}



\begin{proof}(of Theorem \ref{theorem:stochastic})
Let $E$ be the event that, if we throw $n$ balls independently into $n$ bins,
with the $j^{th}$\ ball thrown according to the probability distribution
$\Pr\left(  \text{bin }i\right )  =a_{ij}$, then all $n$ balls land in separate
bins (i.e., there are no collisions). \ Then observe that $\per(A)$\ is simply\ $\Pr\left (E\right) $.

Let $\row_{i}=\left(  a_{ij}\right)  _{j}$\ be the $i^{th}$\ row vector in $A$.
\ Also let $\|\row_{i}\|_{1}:=\sum_{j}a_{ij}$. \ Observe that
$\sum_{i}\|\row_{i}\| _{1}=n$. \ In the balls-in-bins
experiment, let $B_{i}$\ be the number of balls that land in the $i^{th}%
$\ bin. By definition, $\sum_{i}B_{i}=n$ and $\E (B_i)=\|\row_i\|_1$. Also, the event $E$ holds only when $B_{i}=1\forall i$.

Call the $i^{th}$\ row \textit{little} if $\|\row_{i}\|_{1}\leq0.9$, and\ \textit{splittable} if one can partition its entries into
two parts, both of which sum to at least $0.1$. \ Let $L,S\subseteq\left[
n\right]  $ be the sets of little and splittable rows respectively. \ Our
strategy will be to show that

$$\Pr\left( E\right )  \leq\exp\left(
-\Omega\left(  \max\left\{  \left\vert L\right\vert ,\left\vert S\right\vert
\right\}  \right)  \right)  .$$

To begin with the little rows: for each $i\in L$, Markov's inequality implies
that $\Pr\left(  B_{i}\geq1\right)  \leq0.9$. \ Furthermore, the events
$B_{i}\geq1$\ behave as a submartingale, in the sense that, conditioned on
some subset of them occurring, we can only \textit{decrease} the probability
that others occur (by decreasing the expected number of balls available to
land in other bins). \ So by Azuma's inequality, we have%
\[
\Pr\left(  E\right)  \leq\Pr\left(  B_{i}\geq1~\forall i\in L\right)  \leq
\exp\left(  -\frac{\left(  0.1\left\vert L\right\vert \right)  ^{2}%
}{2\left\vert L\right\vert }\right)  =\exp\left(  -\frac{\left\vert
L\right\vert }{200}\right)  ,
\]
where we also used the fact that%
\[
\E \left(  \sum_{i\in L}B_{i}\right)  =\sum_{i\in L}\|\row_{i}\|_{1}\leq0.9\left\vert L\right\vert .
\]

For the splittable rows: for each $i\in S$, we claim that%
\[
\Pr\left(  B_{i}\geq2\right)  \geq\left(  1-e^{-0.1}\right)  ^{2}>0.009.
\]
The reason is that we can partition the $n$ balls into two sets $P$\ and $Q$,
both of which have at least $0.1$\ balls landing in the $i^{th}$\ bin in
expectation. \ Because the balls are thrown independently, this implies that
$P$\ (and likewise, $Q$) must have at least one ball landing in the $i^{th}%
$\ bin with probability at least $1-e^{-0.1}$. \ Moreover, these events are
independent between $P$ and $Q$.

Now, the events $B_{i}\leq1$\ behave as a submartingale: conditioned on some
of them occurring, we can only \textit{decrease} the probability that others
occur, by increasing the expected number of balls available for other bins.
\ So by Azuma's inequality,%
\[
\Pr\left(  E\right)  \leq\Pr\left(  B_{i}\leq1~\forall i\in S\right)  \leq
\exp\left(  -\frac{\left(  0.009\left\vert S\right\vert \right)  ^{2}%
}{2\left\vert S\right\vert }\right)  <\exp\left(  -\frac{\left\vert
S\right\vert }{25000}\right)  .
\]

So, in conclusion, if $\per(A)  =\Pr\left(
E\right)  $ is $n^{-C}$, then $\left\vert L\right\vert
$\ and $\left\vert S\right\vert $\ must both be $O\left( C \log n\right)  $.
\ Now consider a row $i$\ that is neither little nor splittable. \ We have
$\|\row_{i}\|_{1}>0.9$. \ Moreover, $\row_{i}$\ must contain an
entry $j$\ that is at least $0.8$, since otherwise we could split $\row_{i}$,\ by
setting $P=\left\{  j\right\}  $\ and $Q=\left[  n\right]  \setminus\left\{
j\right\}  $.

\end{proof}

\end{document}